\documentclass{article}

\usepackage{
	amsmath, 
	amsthm, 
	amssymb,
	fancyhdr, 
	setspace,
	tikz-cd,
}

\newcommand{\za}{\alpha}
\newcommand{\zb}{\beta}

\newcommand{\zd}{\delta}
\newcommand{\zg}{\gamma}

\usepackage{lscape}
\usepackage{color}
\usepackage[color,matrix,arrow]{xy}
\usepackage{setspace}
\xyoption{all}

\DeclareMathOperator{\Ext}{Ext}
\DeclareMathOperator{\Hom}{Hom}
\DeclareMathOperator{\pd}{pd}
\DeclareMathOperator{\id}{id}
\DeclareMathOperator{\add}{add}

\DeclareMathOperator{\End}{End}
\DeclareMathOperator{\ind}{ind}
\DeclareMathOperator{\coker}{coker}

\usepackage{avant}
\usepackage{amsmath, amssymb, amsfonts, color, tikz, bbm, txfonts, euscript}
\usepackage{amssymb}
\usepackage{enumerate}
\usepackage{bm}
\newtheorem{theorem}{Theorem}[section]
\newtheorem{lemma}[theorem]{Lemma}

\newtheorem{prop}[theorem]{Proposition}
\newtheorem{cor}[theorem]{Corollary}
\theoremstyle{definition}
\newtheorem{mydef}[theorem]{Definition}
\newtheorem{example}[theorem]{Example}

\begin{document}

\thispagestyle{fancy}

\title{PROJECTIVE DIMENSIONS AND EXTENSIONS OF MODULES FROM TILTED TO CLUSTER-TILTED ALGEBRAS}
\author{Stephen Zito\thanks{The author was supported by the University of Connecticut.}}    
\maketitle

\begin{abstract}
We study the module categories of a tilted algebra $C$ and the corresponding cluster-tilted algebra $B=C\ltimes E$ where $E$ is the $C$-$C$-bimodule $\Ext_C^2(DC,C)$.  We investigate how various properties of a $C$-module are affected when considered in the module category of $B$.  We give a complete classification of the projective dimension of a $C$-module inside $\mathop{\text{mod}}{B}$.  If a $C$-module $M$ satisfies $\text{Ext}_C^1(M,M)=0$, we show two sufficient conditions for $M$ to satisfy $\text{Ext}_B^1(M,M)=0$.  In particular, if $M$ is indecomposable and $\text{Ext}_C^1(M,M)=0$, we prove $M$ always satisfies $\text{Ext}_B^1(M,M)=0$. 
\end{abstract}

\section{Introduction} 
We are interested in studying the representation theory of cluster-tilted algebras which are finite dimensional associative algebras that were introduced by Buan, Marsh, and Reiten in $\cite{BMR}$ and, independently, by Caldero, Chapoton, and Schiffler in $\cite{CCS}$ for type $\mathbb{A}$.
\par
One motivation for introducing these algebras came from Fomin and Zelevinsky's cluster algebras $\cite{FZ}$.  Cluster algebras were developed as a tool to study dual canonical bases and total positivity in semisimple Lie groups, and cluster-tilted algebras were constructed as a categorification of these algebras. To every cluster in an acyclic cluster algebra one can associate a cluster-tilted algebra, such that the indecomposable rigid modules over the cluster-tilted algebra correspond bijectively to the cluster variables outside the chosen cluster.  Many people have studied cluster-tilted algebras in this context, see for example $\cite{BBT,BMR,BMR2,BMR3,CCS2,KR}$.
\par
The second motivation came from classical tilting theory.  Tilted algebras are the endomorphism algebras of tilting modules over hereditary algebras, whereas cluster-tilted algebras are the endomorphism algebras of cluster-tilting objects over cluster categories of hereditary algebras.  This similarity in the two definitions lead to the following precise relation between tilted and cluster-tilted algebras, which was established in $\cite{ABS}$.
\par
There is a surjective map
\[
\{\text{tilted}~\text{algebras}\}\longmapsto \{\text{cluster-tilted}~\text{algebras}\}
\]
\[
C\longmapsto B=C\ltimes E
\]
where $E$ denotes the $C$-$C$-bimodule $E=\text{Ext}_C^2(DC,C)$ and $C\ltimes E$ is the trivial extension.
\par
This result allows one to define cluster-tilted algebras without using the cluster category.  It is natural to ask how the module categories of $C$ and $B$ are related and several results in this direction have been obtained, see for example $\cite{ABS2,ABS3,ABS4,BFPPT,BOW}$.  In this work, we investigate how various properties of a $C$-module are affected when the same module is viewed as a $B$-module via the standard embedding.  We let $M$ be a right $C$-module and define a right $B=C\ltimes E$ action on $M$ by 
\[
M\times B\rightarrow M~~,~~(m,(c,e))\mapsto mc.
\]
\par
 Our first main result is on the projective dimension of a $C$-module when viewed as a $B$-module.  Here, $\tau_C^{-1}$ and $\Omega_C^{-1}$ denote respectively the inverse Auslander-Reiten translation and first cosyzygy of a $C$-module.
 \begin{theorem}
 Let C be a tilted algebra, $E=\Ext_C^2(DC,C)$, and $B=C\ltimes E$ the corresponding cluster-tilted algebra.
 \begin{enumerate}
 \item[\emph{(a)}] If $\pd_CM=0$, then $\pd_BM=0$ if and only if $\id_CM\leq1$.  Otherwise, $\pd_BM=\infty$.
 \item[\emph{(b)}] If $\pd_CM=2$, then $\pd_BM=\infty$.
 \item[\emph{(c)}] Let $\pd_CM=1$ with minimal projective resolution $0\rightarrow P_1\rightarrow P_0\rightarrow M\rightarrow 0$.  Then $\pd_BM=1$ if and only if $\id_CM\leq1$ and $\tau_C^{-1}\Omega_C^{-1}P_0\cong\tau_C^{-1}\Omega_C^{-1}P_1$.  Otherwise, $\pd_BM=\infty$.
 \end{enumerate}
 \end{theorem}
 
 Our second main result is on $C$-modules that satisfy $\Ext_C^1(M,M)=0$.  These are known as $\it{rigid}$ modules.  Here, our result holds in a more general setting with $C$ an algebra of global dimension equal to 2.  We determine two sufficient conditions to guarantee when a rigid $C$-module remains rigid when viewed as a $B$-module, i.e., $\text{Ext}_B^1(M,M)=0$.  Here, $\tau_C$ and $\Omega_C$ denote respectively the Auslander-Reiten translation and first syzygy of a $C$-module.
 \begin{theorem}
 Let M be a rigid C-module with a projective cover $P_0\rightarrow M$ and an injective envelope $M\rightarrow I_0$ in $\mathop{\emph{mod}}C$.  
 \begin{enumerate}
 \item[\emph{(a)}] If $\Hom_C(\tau_C^{-1}\Omega_C^{-1}P_0,M)=0$, then M is a rigid B-module.
 \item[\emph{(b)}] If $\emph{Hom}_C(M,\tau_C\Omega_CI_0)=0$, then M is a rigid B-module.
 \end{enumerate}
 \end{theorem} 
 As an immediate consequence, in the case $C$ is tilted, we obtain an affirmative answer to whether an indecomposable rigid $C$-module remains rigid as a $B$-module.
\begin{cor} 
Let C be a tilted algebra with B the corresponding cluster-titled algebra.  Suppose M is an indecomposable, rigid C-module.  Then M is a rigid B-module.
 \end{cor}
 
\section{Notation and Preliminaries}
We now set the notation for the remainder of this paper. All algebras are assumed to be finite dimensional over an algebraically closed field $k$.  Suppose $Q=(Q_0,Q_1)$ is a connected quiver without oriented cycles where $Q_0$ denotes a finite set of vertices and $Q_1$ denotes a finite set of oriented arrows.  By $kQ$ we denote the path algebra of $Q$.  If $\Lambda$ is a $k$-algebra then denote by $\mathop{\text{mod}}\Lambda$ the category of finitely generated right $\Lambda$-modules and by $\mathop{\text{ind}}\Lambda$ a set of representatives of each isomorphism class of indecomposable right $\Lambda$-modules.  Given $M\in\mathop{\text{mod}}\Lambda$, the projective dimension of $M$ in $\mathop{\text{mod}}\Lambda$ is denoted $\pd_{\Lambda}M$ and its injective dimension by $\id_{\Lambda}M$.  We denote by $\add M$ the smallest additive full subcategory of $\mathop{\text{mod}}\Lambda$ containing $M$, that is, the full subcategory of $\mathop{\text{mod}}\Lambda$ whose objects are the direct sums of direct summands of the module $M$.  As mentioned before, we let $\tau_{\Lambda}$ and $\tau^{-1}_{\Lambda}$ be the Auslander-Reiten translations in mod$\Lambda$.  We let $D$ be the standard duality functor $\Hom_k(-,k)$.  Also mentioned before, $\Omega M$ and $\Omega^{-1}M$ will denote the first syzygy and first cosyzygy of $M$.  Finally,  let gl.dim stand for the global dimension of an algebra.  
 
 \subsection{Tilted Algebras}
  Tilting theory is one of the main themes in the study of the representation theory of algebras.  Given a $k$-algebra $A$, one can construct a new algebra $B$ in such a way that the corresponding module categories are closely related.  The main idea is that of a tilting module.
   \begin{mydef} Let $A$ be an algebra.  An $A$-module $T$ is a $\emph{partial tilting module}$ if the following two conditions are satisfied: 
   \begin{enumerate}
   \item[($\text{1}$)] $\pd_AT\leq1$.
   \item[($\text{2}$)] $\Ext_A^1(T,T)=0$.
   \end{enumerate}
   A partial tilting module T is called a $\emph{tilting module}$ if it also satisfies the following additional condition:
   \begin{enumerate}
   \item[($\text{3}$)] There exists a short exact sequence $0\rightarrow A\rightarrow T'\rightarrow T''\rightarrow 0$ in $\mathop{\text{mod}}A$ with $T'$ and $T''$ $\in \add T$.
   \end{enumerate}
   \end{mydef}
Partial tilting modules induce torsion pairs in a natural way.  We consider the restriction to a subcategory $\mathcal{C}$ of a functor $F$ defined originally on a module category, and we denote it by $F|_{\mathcal{C}}$.  Also, let $S$ be a subcategory of a category $C$.  We say $S$ is a $full$ $subcategory$ of $C$ if, for each pair of objects $X$ and $Y$ of $S$, $\Hom_S(X,Y)=\Hom_C(X,Y)$. 
   \begin{mydef} A pair of full subcategories $(\mathcal{T},\mathcal{F})$ of $\mathop{\text{mod}}A$ is called a $\emph{torsion pair}$ if the following conditions are satisfied:
   \begin{enumerate}
   \item[(a)] $\text{Hom}_A(M,N)=0$ for all $M\in\mathcal{T}$, $N\in\mathcal{F}.$
   \item[(b)] $\text{Hom}_A(M,-)|_\mathcal{F}=0$ implies $M\in\mathcal{T}.$
   \item[(c)] $\text{Hom}_A(-,N)|_\mathcal{T}=0$ implies $N\in\mathcal{F}.$
   \end{enumerate}
   \end{mydef}
   Consider the following full subcategories of $\mathop{\text{mod}}A$ where $T$ is a partial tilting module.
 \[
 \mathcal{T}(T)=\{M\in\mathop{\text{mod}}A~|~ \text{Ext}_A^{1}(T,M)=0\}
 \]
 \[
 \mathcal{F}(T)=\{M\in\mathop{\text{mod}}A~|~\text{Hom}_A(T,M)=0\}
 \]
 Then $(\mathcal{T}(T),\mathcal{F}(T))$ is a torsion pair in $\mathop{\text{mod}}A$ called the $\it{induced~torsion~pair}$ of $T$.  Considering the endomorphism algebra $C=\End_AT$, there is an induced torsion pair, $(\mathcal{X}(T),\mathcal{Y}(T))$, in $\mathop{\text{mod}}C$.
 \[
 \mathcal{X}(T)=\{M\in\mathop{\text{mod}}B~|~M\otimes_CT=0\}
 \]
 \[
 \mathcal{Y}(T)=\{M\in\mathop{\text{mod}}B~|~\text{Tor}_1^C(M,T)=0\}
 \]
 
 We now state the definition of a tilted algebra.
 \begin{mydef} Let $A$ be a hereditary algebra with $T$ a tilting $A$-module.  Then the algebra $C=\End_AT$ is called a $\emph{tilted algebra}$.
 \end{mydef}
 
 The following proposition describes several facts about tilted algebras.  Let $A$ be an algebra and $M$, $N$ be two indecomposable $A$-modules.  A $\it{path}$ in $\mathop{\text{mod}}A$ from $M$ to $N$ is a sequence 
 \[
 M=M_0\xrightarrow{f_1}M_1\xrightarrow{f_2}M_2\rightarrow\dots\xrightarrow{f_s}M_s=N
 \]
 where $s\geq0$, all the $M_i$ are indecomposable, and all the $f_i$ are nonzero nonisomorphisms.  In this case, $M$ is called a $\it{predecessor}$ of $N$ in $\mathop{\text{mod}}A$ and $N$ is called a $\it{successor}$ of $M$ in $\mathop{\text{mod}}A$.
 \begin{prop}$\emph{\cite[VIII,~Lemma~3.2.]{ASS}}$.
 \label{Tilting Prop}
 Let $A$ be a hereditary algebra, $T$ a tilting $A$-module, and $C=\End_AT$ the corresponding tilted algebra.  Then
 \begin{enumerate}
 \item[$\emph{(a)}$] $\mathop{\emph{gl.dim}}C\leq2$.
 \item[$\emph{(b)}$] For all $\text{M}\in\ind C$, $\id_CM\leq1$ or $\pd_CM\leq1$.
 \item[$\emph{(c)}$] For all $\text{M}\in\mathcal{X}(T)$, $\id_CM\leq1$.
 \item[$\emph{(d)}$] For all $\text{M}\in\mathcal{Y}(T)$, $\pd_CM\leq1$.
 \item[$\emph{(e)}$] $(\mathcal{X}(T),\mathcal{Y}(T))$ is splitting, which means that every indecomposable C-module belongs to either $\mathcal{X}(T)$ or $\mathcal{Y}(T)$.
 \item[$\emph{(f)}$] $\mathcal{Y}(T)$ is closed under predecessors and $\mathcal{X}(T)$ is closed under successors.
 \end{enumerate}
 \end{prop}
 
 \subsection{Cluster categories and cluster-tilted algebras} Let $A=kQ$ and let $\mathcal{D}^b(\mathop{\text{mod}}A)$ denote the derived category of bounded complexes of $A$-modules as summarized in $\cite{BMMRRT}$.  The $cluster$ $category$ $\mathcal{C}_A$ is defined as the orbit category of the derived category with respect to the functor $\tau_\mathcal{D}^{-1}[1]$, where $\tau_\mathcal{D}$ is the Auslander-Reiten translation in the derived category and $[1]$ is the shift.  Cluster categories were introduced in $\cite{BMMRRT}$, and in $\cite{CCS}$ for type $\mathbb{A}$, and were further studied in $\cite{A,K,KR,P}$.  They are triangulated categories $\cite{K}$, that are 2-Calabi Yau and have Serre duality $\cite{BMMRRT}$.
 
 \par
 An object $T$ in $\mathcal{C}_A$ is called $cluster$-$tilting$ if $\text{Ext}_{\mathcal{C}_A}^1(T,T)=0$ and $T$ has $|Q_0|$ non-isomorphic indecomposable direct summands.  The endomorphism algebra $\End_{\mathcal{C}_A}T$ of a cluster-tilting object is called a $cluster$-$tilted$ $algebra$ $\cite{BMR}$.
 \par
 The following theorem was shown in $\cite{KR}$.  It characterizes the homological dimensions of a cluster-tilted algebra.
 \begin{theorem}$\emph{\cite{KR}}$.
 \label{Gorenstein}
   Cluster-tilted algebras are 1-Gorenstein, that is, every projective module has injective dimension at most 1 and every injective module has projective dimension at most 1.
 \end{theorem}
 As an important consequence, the projective dimension and the injective dimension of any module in a cluster-tilted algebra are simultaneously either infinite, or less than or equal to 1 (see $\cite[\text{Section~2.1}]{KR}$).

 \subsection{Relation Extensions}  Let $C$ be an algebra of global dimension at most 2 and let $E$ be the $C$-$C$-bimodule $E=\text{Ext}_C^2(DC,C)$.  
 \begin{mydef} The $\emph{relation extension}$ of $C$ is the trivial extension $B=C\ltimes E$, whose underlying $C$-module structure is $C\oplus E$, and multiplication is given by $(c,e)(c^{\prime},e^{\prime})=(cc^{\prime},ce^{\prime}+ec^{\prime})$.
 \end{mydef}  
 
 Relation extensions were introduced in $\cite{ABS}$.  In the special case where $C$ is a tilted algebra, we have the following result.
 \begin{theorem}$\emph{\cite{ABS}}$.  Let C be a tilted algebra.  Then $B=C\ltimes\emph{Ext}_C^2(DC,C)$ is a cluster-titled algebra.  Moreover all cluster-tilted algebras are of this form.
 \end{theorem}
 
  \subsection{Induction and coinduction functors}  A fruitful way to study cluster-tilted algebras is via induction and coinduction functors.  Recall, $D$ denotes the standard duality functor.
 
 \begin{mydef}
 Let $C$ be a subalgebra of $B$ such that $1_C=1_B$, then
 \[
 -\otimes_CB:\mathop{\text{mod}}C\rightarrow\mathop{\text{mod}}B
 \]
 is called the $\emph{induction functor}$, and dually
 \[
 D(B\otimes_CD-):\mathop{\text{mod}}C\rightarrow\mathop{\text{mod}}B
 \]
 is called the $\emph{coinduction functor}$.  Moreover, given $M\in\mathop{\text{mod}}C$, the corresponding induced module is defined to be $M\otimes_CB$, and the coinduced module is defined to be $D(B\otimes_CDM)$.  
 \end{mydef}
 We can say more in the situation when $B$ is a split extension of $C$.  Call a $C$-$C$-bimodule $E$ $\it{nilpotent}$ if, for $n\geq0$, $E\otimes_CE\otimes_C\dots\otimes_CE=0$, where the tensor product is performed $n$ times.
 
 \begin{mydef}
 
 Let $B$ and $C$ be two algebras.  We say $B$ is a $\emph{split extension}$ of $C$ by a nilpotent bimodule $E$ if there exists a short exact sequence of $B$-modules
 \[
 0\rightarrow E\rightarrow B\mathop{\rightleftarrows}^{\mathrm{\pi}}_{\mathrm{\sigma}} C\rightarrow 0
\]
where $\pi$ and $\sigma$ are algebra morphisms, such that $\pi\circ\sigma=1_C$, and $E=\ker\pi$ is nilpotent.  
\end{mydef}
In particular, relation extensions are split extensions.  The next proposition shows a precise relationship between a given $C$-module and its image under the induction and coinduction functors.
\begin{prop}$\emph{\cite[Proposition~3.6]{SS}}$.  
\label{SS, Proposition 3.6}
Suppose B is a split extension of C by a nilpotent bi-module E.  Then, for every $M\in\mathop{\emph{mod}}C$, there exists two short exact sequences of B-modules:
\begin{enumerate}
\item[$\emph{(a)}$] $0\rightarrow M\otimes_CE\rightarrow M\otimes_CB\rightarrow M\rightarrow 0$
\item[$\emph{(b)}$] $0\rightarrow M\rightarrow D(B\otimes_CDM)\rightarrow D(E\otimes_CDM)\rightarrow 0$
\end{enumerate}

\end{prop}
 The next two results give information on the projective cover and the minimal projective presentation of an induced module.
 \begin{lemma}$\emph{\cite[Lemma~1.3]{AM}}$.  
 \label{projective cover}
 Suppose B is a split extension of C by a nilpotent bimodule E.  Let $M$ be a $C$-module.  If $f\!:P\rightarrow M$ is a projective cover in $\mathop{\emph{mod}}C$, then $f\otimes_C1_B\!:P\otimes_CB\rightarrow M\otimes_CB$ is a projective cover in $\mathop{\emph{mod}}B.$
 \end{lemma}
 \begin{lemma}$\emph{\cite{AM}}$.  
 \label{Projective Resolution}
 Suppose B is a split extension of C by a nilpotent bimodule E.  Let $M$ be a $C$-module.  If $P_1\rightarrow P_0\rightarrow M\rightarrow 0$ is a projective presentation, then $P_1\otimes_CB\rightarrow P_0\otimes_CB\rightarrow M\otimes_CB\rightarrow 0$ is a projective presentation.  Furthermore, if the first is minimal, then so is the second.
 \end{lemma}
 The following is a crucial result needed in section 3.
 \begin{lemma}$\emph{\cite[Lemma~2.2]{AM}}$.  
 \label{AM, Projective}
 For a C-module M, we have $\pd_B(M\otimes_CB)\leq1$ if and only if $\pd_CM\leq1$ and $\Hom_C(DE, \tau_CM)=0$.
 \end{lemma}
  
  \subsection{Standard results}
 In this subsection we list several standard results which hold over arbitrary $k$-algebras of finite dimension.  We begin with a result on the projective dimension of arbitrary modules related by a short exact sequence.
 \begin{lemma}$\emph{\cite[Appendix,~Proposition~4.7]{ASS}}$. 
 \label{Standard Projective Restrictions}
 Let A be a finite dimensional $k$-algebra and suppose $0\rightarrow L\rightarrow M\rightarrow N\rightarrow 0$ is a short exact sequence in $\mathop{\emph{mod}}A$.
 \begin{enumerate}
 \item[$(\emph{a})$] $\pd_AN\leq \emph{max}(\pd_AM,1+\pd_AL)$, and equality holds if $\pd_AM\neq\pd_AL.$
 \item[$(\emph{b})$] $\pd_AL\leq \emph{max}(\pd_AM,-1+\pd_AN)$, and equality holds if $\pd_AM\neq\pd_AN.$
 \item[$(\emph{c})$] $\pd_AM\leq\emph{max}(\pd_AL,\pd_AN)$, and equality holds if $\pd_AN\neq 1+\pd_AL.$
 
  \end{enumerate}
 \end{lemma}
 The next result, which relates the Ext and Tor functors, will be needed in section 3.
 \begin{prop}$\emph{\cite[Appendix,~Proposition~4.11]{ASS}}$
 \label{Tor}
 Let A be a finite dimensional k-algebra.  For all modules Y and Z in $\mathop{\emph{mod}}A$, we have
 \[
 D\emph{Ext}_A^1(Y,DZ)\cong\emph{Tor}_1^A(Y,Z).
 \]
 \end{prop}
 The following lemma is well known.
 \begin{lemma}$\emph{\cite[IV,~Lemma~2.7]{ASS}}$
 \label{Blue 115}
 Let A be a finite dimensional k-algebra and M an A-module.
\begin{enumerate}
 \item[$\emph{(a)}$]$\pd_AM\leq1$ if and only if $\Hom_A(DA,\tau_AM)=0$
 \item[$\emph{(b)}$]$\id_AM\leq1$ if and only if $\Hom_A(\tau_A^{-1}M,A)=0$
 \end{enumerate}
 \end{lemma}
For our next two statements we need two definitions.  We say a a submodule $S$ of a module $M$ is $\it{superfluous}$ if, whenever $L\subseteq M$ is a submodule with $L+S=M$, then $L=M$.  An epimorphism $f\!:M\rightarrow N$ is $\it{minimal}$ if $\ker f$ is superfluous in $M$.  In particular, any projective cover is minimal.
 \begin{lemma}$\emph{\cite[I,~Lemma~5.6]{ASS}}$
 \label{WTF}
 Let $A$ be a finite dimensional k-algebra and $M$ an $A$-module.  Then an epimorphism $f\!:P\rightarrow M$ is minimal if and only if for any morphism $g\!:N\rightarrow P$, the surjectivity of $f\circ g$ implies the surjectivity of $g$. 
 \end{lemma}
 
\begin{cor}
 \label{Superfluous}
 If $g\!:M\rightarrow N$ and $f\!:N\rightarrow L$ are epimorphisms and $f$ and $g$ are minimal, then $f\circ g$ is minimal.
 \end{cor}
 \begin{proof}
 Cleary, $f\circ g$ is surjective.  Thus, we must show that $\ker f\circ g$ is superfluous.  Let $h\!:X\rightarrow M$ be a morphism such that $f\circ g\circ h$ is surjective.  Since $f\circ g\circ h=f\circ(g\circ h)$ and $f$ is minimal, we know by Lemma $\ref{WTF}$ that $g\circ h$ is surjective.  Since $g$ is minimal, we may use Lemma $\ref{WTF}$ again to say $h$ is surjective.  Thus, $f\circ g\circ h$ is surjective and a final application of Lemma $\ref{WTF}$ says that $f\circ g\circ h$ is minimal.  
\end{proof}

 \subsection{Induced and coinduced modules in cluster-tilted algebras} In this section we cite several properties of the induction and coinduction functors particularly when $C$ is an algebra of global dimension at most 2 and $B=C\ltimes E$ is the trivial extension of $C$ by the $C$-$C$-bimodule $E=\text{Ext}_C^2(DC,C)$.  In the specific case when $C$ is also a tilted algebra, $B$ is the corresponding cluster-tilted algebra.
 \begin{prop}$\emph{\cite[Proposition~4.1]{SS}}$. 
 \label{SS Prop 4.1}
 Let C be an algebra of global dimension at most 2.  Then
 \begin{enumerate}
 \item[$\emph{(a)}$] $E\cong \tau_C^{-1}\Omega_C^{-1}C$.
 \item[$\emph{(b)}$] $DE\cong \tau_C\Omega_C DC$.
 \item[$\emph{(c)}$] $M\otimes_C E\cong\tau_C^{-1}\Omega_C^{-1}M$.
 \item[$\emph{(d)}$] $D(E\otimes_C DM)\cong\tau_C\Omega_C M$.
 
 \end{enumerate}
 \end{prop}
 
 The next two results use homological dimensions to extract information about induced and coinduced modules.
 \begin{prop}$\emph{\cite[Proposition~4.2]{SS}}$.  
 \label{SS, Proposition 4.2}
 Let C be an algebra of global dimension at most 2, and let $B=C\ltimes E$.  Suppose $M\in\mathop{\emph{mod}}C$, then
 \begin{enumerate}
 \item[$\emph{(a)}$] $\id_CM\leq1$ if and only if $M\otimes_CB\cong M$.
 \item[$\emph{(b)}$] $\pd_CM\leq1$ if and only if $D(B\otimes_CDM)\cong M$.
 \end{enumerate}
 \end{prop}
  \begin{lemma}$\emph{\cite[Lemma~4.4]{SS}}$.  
 \label{SS, Lemma 4.4}
 Let C be an algebra of global dimension 2 and $M$ a $C$-module.
 \begin{enumerate}
 \item[$\emph{(a)}$] $\pd_CN=2$ for all nonzero $N\in\mathop{\emph{add}}(M\otimes_C E)$.
 \item[$\emph{(b)}$] $\id_CN=2$ for all nonzero $N\in\mathop{\emph{add}}(D(E\otimes_C DM))$.
 \end{enumerate}
 \end{lemma} 
We end this section with a lemma which tells us what the projective cover of a projective $C$-module is in $\mathop{\text{mod}}B$.
\begin{lemma}$\emph{\cite[Lemma~2.7]{ABS}}$
\label{Last One}
Let C be an algebra of global dimension at most 2 and $B=C\ltimes E$.  Suppose P is a projective C-module.  Then the induced module, $P\otimes_CB$, is a projective cover of P in $\mathop{\emph{mod}}B$.
\end{lemma}

We also have the following important fact.
 \begin{lemma}$\emph{\cite[Corollary~1.2]{AZ}}$.
 \label{AR submodule}
 $\tau_CM$ and $\tau_B(M\otimes_CB)$ are submodules of $\tau_BM$.
 \end{lemma}

 \section{Homological Dimensions}
 In this section let $C$ be an algebra of global dimension 2, $E=\text{Ext}_C^2(DC,C)$, and $B=C\ltimes E$ be the relation extension.  We investigate what happens to the projective dimension of a $C$-module $M$ when viewed as a $B$-module.  In the special case when $C$ is a tilted algebra and $B$ is the corresponding cluster-tilted algebra, we provide a complete classification.  First, we prove a lemma which provides a useful criteria for a $C$-module to have projective or injective dimension at most 1 in an algebra of global dimension 2.
 \begin{lemma} 
 \label{Homological Result}
 Let M be a C-module.  Then, 
 \begin{enumerate}
 \item[$\emph{(a)}$] $\pd_CM\leq1$ if and only if $\emph{Hom}_C(\tau_C^{-1}\Omega_C^{-1}C,M)=0.$
 \item[$\emph{(b)}$] $\id_CM\leq1$ if and only if $\emph{Hom}_C(M,\tau_C\Omega_C DC)=0.$
 \end{enumerate}
 \end{lemma}
 \begin{proof}
 We prove (a) with the proof of (b) being similar.  Assume $\pd_CM\leq1$.  Consider the short exact sequence 
 \[
 0\rightarrow C\rightarrow I_0\rightarrow \Omega_C^{-1}C\rightarrow 0
 \]
 where $I_0$ is an injective envelope of $C$.  Apply Hom$_C(M,-)$ to obtain an exact sequence
 \[
 \text{Ext}_C^1(M,I_0)\rightarrow\text{Ext}_C^1(M,\Omega_C^{-1}C)\rightarrow\text{Ext}_C^2(M,C).  
 \]
 Now, $\text{Ext}_C^1(M,I_0)=0$ because $I_0$ is injective and $\text{Ext}_C^2(M,C)=0$ because $\pd_CM\leq1$.  Since the sequence is exact, $\Ext_C^1(M,\Omega_C^{-1}C)=0$.  By the Auslander-Reiten formulas, $D{\text{Hom}}_C(\tau_C^{-1}\Omega_C^{-1}C,M)\cong\text{Ext}_C^1(M,\Omega_C^{-1}C).$  Thus, 
 \[
 0=\text{Ext}_C^1(M,\Omega_C^{-1}C)\cong D\text{Hom}_C(\tau_C^{-1}\Omega_C^{-1}C,M).  
 \]
 \par
 Conversely, assume Hom$_C(\tau_C^{-1}\Omega_C^{-1}C,M)=0$.  Then we have 
 \[
 D\underline{\text{Hom}}_C(\tau_C^{-1}\Omega_C^{-1}C,M)\cong\text{Ext}_C^1(M,\Omega_C^{-1}C)=0 
 \]
 by the Auslander-Reiten formulas.  We then have $\text{Ext}_C^2(M,C)\cong\text{Ext}_C^1(M,\Omega_C^{-1}C)=0$.  Since C has global dimension equal to 2, this implies $\pd_CM\leq1$.
 \end{proof}

 We begin with the case where $M$ is a projective $C$-module.
 \begin{prop} 
 \label{0 case pd}
 Let M be a projective C-module. Then $\pd_BM=0$ if and only if $\id_CM\leq1$. 
 \end{prop}
 \begin{proof}
 Assume $\pd_BM=0$.  By Proposition $\ref{SS, Proposition 3.6}$ we have a short exact sequence 
 \[
 0\rightarrow\tau_C^{-1}\Omega_C^{-1}M\rightarrow M\otimes_CB\rightarrow M\rightarrow0
 \]
 where $M\otimes_CB$ is a projective cover by Lemma $\ref{projective cover}$.  This implies $M\otimes_CB\cong M$ and $\tau_C^{-1}\Omega_C^{-1}M=0$.  By Proposition $\ref{SS, Proposition 4.2}$, we conclude $\id_CM\leq1$.
 \par
 Conversely, assume $\id_CM\leq1$.  Then Proposition $\ref{SS, Proposition 4.2}$ implies $M\otimes_CB\cong M$ and we conclude $M$ is a projective $B$-module.
 \end{proof}
 The case where the projective dimension of $M$ is equal to 2 holds in a more general setting which we explicitly state.
 \begin{prop} 
 \label{2 case pd}
 Let C be an algebra of global dimension 2 with B a split extension by a nilpotent bimodule E.  If M is a C-module with $\pd_CM=2$, then $\pd_BM\geq2$.
 \end{prop}
 \begin{proof}
 By Lemma $\ref{AM, Projective}$, we have $\pd_B(M\otimes_CB)\geq2$.  This implies the existence of a non-zero morphism $f\!:DB\rightarrow\tau_B(M\otimes_CB)$ by Lemma $\ref{Blue 115}$.  By Lemma $\ref{AR submodule}$, we have an injective morphism $i\!:\tau_B(M\otimes_CB)\rightarrow\tau_BM$.  Thus, there is a non-zero morphism $i\circ f\!:DB\rightarrow\tau_BM$.  By Lemma $\ref{Blue 115}$ again, we have $\pd_BM\geq1$.   
 \end{proof}
 The case where the projective dimension of $M$ is equal to 1 is the most restrictive.  
 \begin{prop} 
  \label{1 case pd}
  Let M be a C-module with $\pd_CM=1$ and a minimal projective resolution $0\rightarrow P_1\rightarrow P_0\rightarrow M\rightarrow 0$ in $\mathop{\emph{mod}}C$.  Then $\id_CM\leq1$ and $\tau_C^{-1}\Omega^{-1}P_1\cong\tau_C^{-1}\Omega^{-1}P_0$ if and only if $\pd_BM=1$.
 \end{prop}
 \begin{proof}
 Assume $\id_CM\leq1$ and $\tau_C^{-1}\Omega_C^{-1}P_1\cong\tau_C^{-1}\Omega_C^{-1}P_0$.  Since $\id_CM\leq1$, by  Proposition $\ref{SS, Proposition 4.2}$, we have $M\otimes_CE\cong\tau_C^{-1}\Omega_C^{-1}M=0$ and $M\otimes_CB\cong M$.  Using Lemma $\ref{AM, Projective}$, we need to to show $\text{Hom}_C(DE,\tau_CM)=0$.  Apply ${-}\otimes_CE$ to the minimal projective resolution of $M$ to obtain the exact sequence 
 \begin{equation}\tag{1}
 \text{Tor}_1^C(P_1,E)\rightarrow\text{Tor}_1^C(M,E)\rightarrow P_1\otimes_CE\rightarrow P_0\otimes_CE\rightarrow M\otimes_CE\rightarrow 0.  
 \end{equation}
 Now, $\text{Tor}_1^C(P_1,E)=0$ because $P_1$ is projective and we showed $M\otimes_CE=0$.  Also, Proposition $\ref{SS, Proposition 4.2}$ says $P_1\otimes_CE\cong\tau_C^{-1}\Omega_C^{-1}P_1\cong\tau_C^{-1}\Omega_C^{-1}P_0\cong P_0\otimes_CE.$  Since (1) is exact, we know $\text{Tor}_1^C(M,E)=0$.  By Proposition $\ref{Tor}$ and the Auslander-Reiten formulas, we have 
 \[
 0=\text{Tor}_1^C(M,E)\cong D\text{Ext}_C^1(M,DE)\cong\overline{\text{Hom}}_C(DE, \tau_CM).  
 \]
 Since $\pd_CM=1$ by assumption, we may use Lemma $\ref{Blue 115}$ and the Auslander-Reiten formulas to say 
 \[
 0=\overline{\text{Hom}}_C(DE,\tau_CM)\cong\text{Hom}_C(DE,\tau_CM).
\]
\par
 Conversely, assume $\pd_BM=1$.  If $\pd_B(M\otimes_CB)>1$ then we have a non-zero composition of morphisms, $DB\rightarrow\tau_B(M\otimes_CB)\rightarrow\tau_BM$, guaranteed by Lemma $\ref{Blue 115}$ and Lemma $\ref{AR submodule}$ .  By Lemma $\ref{Blue 115}$, this contradicts $\pd_BM=1$.  Thus, $\pd_B(M\otimes_CB)=1$ and Proposition $\ref{Tor}$, Lemma $\ref{AM, Projective}$, the Auslander-Reiten formulas, and Lemma $\ref{Blue 115}$ imply 
 \[
 0=\text{Hom}_C(DE,\tau_CM)\cong D\text{Ext}_C^1(M,DE)\cong\text{Tor}_1^C(M,E).  
 \]
 Next, consider the short exact sequence of Propositions $\ref{SS, Proposition 3.6}$ and Proposition $\ref{SS Prop 4.1}$  
 \[
 0\rightarrow\tau_C^{-1}\Omega_C^{-1}M\rightarrow M\otimes_CB\rightarrow M\rightarrow 0
 \]
  in $\mathop{\text{mod}}B$.  Since $\pd_B(M\otimes_CB)$ and $\pd_BM$ are equal to 1, we know Lemma $\ref{Standard Projective Restrictions}$ implies $\pd_B(\tau_C^{-1}\Omega_C^{-1}M)\leq1$.  By Lemma $\ref{SS, Lemma 4.4}$, we know $\pd_C(\tau_C^{-1}\Omega_C^{-1}M)=2$ or $\tau_C^{-1}\Omega_C^{-1}M=0$.  However, Proposition $\ref{2 case pd}$ implies $\text{pd}_B(\tau_C^{-1}\Omega_C^{-1}M)\geq2$.  Thus, $\tau_C^{-1}\Omega_C^{-1}M=0$ and $M\otimes_CB\cong M$.  Returning to sequence (1), since $M\otimes_CB\cong M$ we have $M\otimes_CE=0$.  Also, we have shown that $\text{Tor}_1^C(M,E)=0$.  Since the sequence is exact, we have $P_1\otimes_CE\cong P_0\otimes_CE$ and Proposition $\ref{SS Prop 4.1}$ implies $\tau_C^{-1}\Omega_C^{-1}P_1\cong\tau_C^{-1}\Omega_C^{-1}P_0$.  Finally, since $M\otimes_CE=0$, Proposition $\ref{SS, Proposition 4.2}$ tells us that $\id_CM\leq1$.    
 \end{proof}

 If $M$ is a $C$-module which satisfies the conditions of Proposition $\ref{1 case pd}$, then the following corollary tells us what a minimum projective resolution is in $\mathop{\text{mod}}B$.
 \begin{cor}
 \label{Resolution}
 Let M be a C-module with minimal projective resolution 
 \[
 0\rightarrow P_1\xrightarrow{f_1} P_0\xrightarrow{f_0} M\rightarrow 0.
 \]
 If $\pd_BM=1$, then $0\rightarrow P_1\otimes_CB\xrightarrow{f_1\otimes_C1_B} P_0\otimes_CB\xrightarrow{f_0\otimes_C1_B} M\rightarrow 0$ is a minimal projective resolution in $\mathop{\emph{mod}}B$.
 \end{cor}
 \begin{proof}
 By Lemma $\ref{Projective Resolution}$, we know that $P_1\otimes_CB\xrightarrow{f_1\otimes_C1_B} P_0\otimes_CB\xrightarrow{f_0\otimes_C1_B} M\otimes_CB\rightarrow 0$ is a minimal projective presentation of $M\otimes_CB$ in $\mathop{\text{mod}}B$.  By Proposition $\ref{1 case pd}$, we know $\id_CM\leq1$.  By Proposition $\ref{SS, Proposition 4.2}$ we have $M\otimes_CB\cong M$ and our statement follows.
 \end{proof}
 In the situation where $C$ is an algebra of global dimension 2 and $B$ is a split extension by a nilpotent bimodule $E$, we prove that the global dimension of $B$ is strictly greater then the global dimension of $C$.  We need a lemma.
 \begin{lemma}
 \label{whatever}
 Let M be a projective $C$-module such that $\id_CM=2$.  Then 
 \[
 \pd_BM=\pd_B(\tau_C^{-1}\Omega_C^{-1}M)+1\geq3.
 \]
 \end{lemma}
 \begin{proof}
 Consider the short exact sequence $0\rightarrow \tau_C^{-1}\Omega_C^{-1}M\rightarrow M\otimes_CB\rightarrow M\rightarrow 0$ guaranteed by  Proposition $\ref{SS, Proposition 3.6}$ and Proposition $\ref{SS Prop 4.1}$.  We have $\text{pd}_B(M\otimes_CB)=0$ and $\pd_B(\tau_C^{-1}\Omega_C^{-1}M)\geq2$ by  Proposition $\ref{2 case pd}$.  Our statement then follows from Lemma $\ref{Standard Projective Restrictions}$.
 \end{proof}
 \begin{cor}
 Let C be an algebra of global dimension 2 and B a split extension by a nilpotent bimodule E.  Then $\mathop{\emph{gl.dim.}}B > \mathop{\emph{gl.dim.}}C.$
 \end{cor}
 \begin{proof}
 This follows immediately from Lemma $\ref{whatever}$.
 \end{proof}
 We conclude this section with a complete classification of the projective dimension of a $C$-module when viewed as a $B$-module in the special case $C$ is tilted and $B$ is the corresponding cluster-titled algebra.
 \begin{theorem}
 \label{Main Homo}
 Let C be a tilted algebra, $E=\emph{Ext}_C^2(DC,C)$, and $B=C\ltimes E$ the corresponding cluster-tilted algebra.
 \begin{enumerate}
 \item[$\emph{(a)}$] If $\pd_CM=0$, then $\pd_BM=0$ if and only if $\id_CM\leq1$.  Otherwise, $\pd_BM=\infty$.
 \item[$\emph{(b)}$] If $\pd_CM=2$, then $\pd_BM=\infty$.
 \item[$\emph{(c)}$] Let $\pd_CM=1$ with minimal projective resolution $0\rightarrow P_1\rightarrow P_0\rightarrow M\rightarrow 0$.  Then $\pd_BM=1$ if and only if $\id_CM\leq1$ and $\tau_C^{-1}\Omega_C^{-1}P_0\cong\tau_C^{-1}\Omega_C^{-1}P_1$.  Otherwise, $\pd_BM=\infty$.
 \end{enumerate}
 \end{theorem}
 \begin{proof}
 Part (a) follows from Proposition $\ref{0 case pd}$.  If the conditions for $M$ are not met, then Lemma $\ref{whatever}$ and the 1-Gorenstein property of a cluster-tilted algebra, (Theorem $\ref{Gorenstein}$), shows $\pd_BM=\infty$.  Part (b) follows from  Proposition $\ref{2 case pd}$ and the 1-Gorenstein property.  Finally, part (c) follows from Proposition $\ref{1 case pd}$ and the 1-Gorenstein property. 
 \end{proof}
 For an illustration of this theorem, see Examples $\ref{Homofirst}$, $\ref{Homosecond}$, and $\ref{Homothird}$ in section 5.
 \section{Extensions}
 In this section, we study $C$-modules which have no self-extension, i.e., Ext$_C^1(M,M)=0$.  These modules are typically referred to as rigid modules.  We investigate under what conditions does a rigid $C$-module remain a rigid $B$-module.  Unless otherwise stated, we assume that $C$ is an algebra of global dimension 2 and $B=C\ltimes E$ is a split extension by a nilpotent bimodule $E$.  To prove our main result we first need an easy lemma.  We recall from Lemma $\ref{Last One}$ that if $P$ is a projective $C$-module, then $P\otimes_CB$ is a projective cover of $P$ in $\mathop{\text{mod}}B$. 
 \begin{lemma}
 \label{My Projective Cover}
 Let $M$ be a $C$-module with $f\!:P_0\rightarrow M$ a projective cover in $\mathop{\emph{mod}}C$.  Suppose $g\!:P_0\otimes_CB\rightarrow P_0$ is a projective cover of $P_0$ in $\mathop{\emph{mod}}B$.  Then $f\circ g\!:P_0\otimes_CB\rightarrow M$ is a projective cover of $M$ in $\mathop{\emph{mod}}B$.
 \end{lemma}
 \begin{proof}
 Clearly, $f\circ g$ is surjective.  Thus, we need to show $\ker f\circ g$ is superfluous.  This follows easily from Corollary $\ref{Superfluous}$ since $f$ and $g$ are both minimal.
 \end{proof} 

\begin{theorem} 
 \label{Main Ext}
 Let M be a rigid C-module with a projective cover $P_0\rightarrow M$ and an injective envelope $M\rightarrow I_0$ in $\mathop{\emph{mod}}C$.  
 \begin{enumerate}
 \item[$\emph{(a)}$] If $\emph{Hom}_C(\tau_C^{-1}\Omega^{-1}P_0,M)=0$, then M is a rigid B-module.
 \item[$\emph{(b)}$] If $\emph{Hom}_C(M,\tau_C\Omega I_0)=0$, then M is a rigid B-module.
 \end{enumerate}
 \end{theorem}
 \begin{proof}
 We prove case (a) with case (b) being dual.  In $\mathop{\text{mod}}B$, consider the following short exact sequence of $M$
 \[
 0\rightarrow\Omega_B^1M\xrightarrow{f}P_0\otimes_CB\rightarrow M\rightarrow 0.
 \]
 Apply $\text{Hom}_B(-,M)$ to obtain
 \begin{equation}\tag{1}
 0\rightarrow\text{Hom}_B(M,M)\rightarrow\text{Hom}_B(P_0\otimes_CB,M)\xrightarrow{\overline{f}}\text{Hom}_B(\Omega_B^1M,M)\rightarrow\text{Ext}_B^1(M,M)\rightarrow\ 0.
 \end{equation}
 Since (1) is exact, we need to show that $\overline{f}$ is surjective.  This will imply that $\text{Ext}_B^1(M,M)=0$.  In $\mathop{\text{mod}}C$, consider the sequence
 \[
 0\rightarrow\Omega_C^1M\xrightarrow{g} P_0\xrightarrow{a} M\rightarrow 0.
 \]
 Apply $\text{Hom}_C(-,M)$ to obtain 
 \begin{equation}\tag{2}
 0\rightarrow\text{Hom}_C(M,M)\rightarrow\text{Hom}_C(P_0,M)\xrightarrow{\overline{g}}\text{Hom}_C(\Omega_C^1M,M)\rightarrow\text{Ext}_C^1(M,M).
 \end{equation}
 Since $M$ is a rigid $C$-module by assumption and (2) is exact, we have $\overline{g}$ is surjective.  Next, in $\mathop{\text{mod}}B$, consider the following commutative diagram guaranteed by Lemma $\ref{My Projective Cover}$ and the universal property of the kernel.
 \begin{equation}\tag{3}
\begin{tikzcd}
0 \arrow{r} & \Omega_B^1M \arrow{d}{z} \arrow{r}{f} & P_0\otimes_CB \arrow{d}{w} \arrow{r}{a\circ w} & M \arrow{r} \arrow{d}{id} & 0\\
0 \arrow{r} & \Omega_C^1M \arrow{r}{g} & P_0 \arrow{r}{a} & M \arrow{r} & 0.
\end{tikzcd} 
 \end{equation}
Here, $\it{id}$ is the identity map, $w$ is a projective cover of $P_0$, and $z$ is induced by the universal property of the kernel.  By the Snake Lemma, we know $\ker z\cong\ker w$.  Thus, Proposition $\ref{SS, Proposition 3.6}$ and Proposition $\ref{SS Prop 4.1}$ implies that $\ker z\cong\ker w\cong\tau_C^{-1}\Omega_C^{-1}P_0$ .  Thus, we have an exact sequence
\[
0\rightarrow\tau_C^{-1}\Omega_C^{-1}P_0\xrightarrow{i}\Omega_B^1M\xrightarrow{z}\Omega_C^1M\rightarrow\coker z\rightarrow 0.
\]
 Since the morphism $w$ is surjective and $id$ is clearly injective, we may use the Snake Lemma again to say that $\coker z=0$.  Apply $\text{Hom}_B(-,M)$ to obtain an exact sequence
\begin{equation}
0\rightarrow\text{Hom}_B(\Omega_C^1M,M)\xrightarrow{\overline{z}}\text{Hom}_B(\Omega_B^1M,M)\rightarrow\text{Hom}_B(\tau_C^{-1}\Omega_C^{-1}P_0,M).
\end{equation}
Since $\text{Hom}_B(\tau_C^{-1}\Omega_C^{-1}P_0,M)=0$ by assumption, we have that $\overline{z}$ is an isomorphism.  To show $\overline{f}$ is surjective, let $h\in\text{Hom}_B(\Omega_B^1M,M)$.  Since $\overline{z}$ is an isomorphism, we know there exists a morphism $j\in\text{Hom}_B(\Omega_C^1M,M)$ such that $h=j\circ z$.  
\[
\begin{tikzcd}
\Omega_B^1M\arrow{dr}{h=j\circ z}\arrow{d}{z}\\
\Omega_C^1M\arrow{r}{j} & M
\end{tikzcd}
\]
Since $\overline{g}$ is surjective, there exists a morphism $l\in\text{Hom}_B(P_0,M)$ such that $j=l\circ g$.  
\[
\begin{tikzcd}
\Omega_C^1M\arrow{r}{j=l\circ g}\arrow{d}{g} & M\\
P_0\arrow{ur}{l}
\end{tikzcd}
\]
Thus, we have $h=l\circ g\circ z$. 
\[
\begin{tikzcd}
\Omega_B^1M\arrow{d}{z}\arrow{ddr}{h=l\circ g\circ z}\\
\Omega_C^1\arrow{d}{g}\\
P_0\arrow{r}{l} & M
\end{tikzcd}
\]
From our commutative diagram (3), we know $g\circ z=w\circ f$.  Thus, we have the following commutative diagram.  
\[
\begin{tikzcd}
\Omega_B^1M\arrow{d}{f}\arrow{ddr}{h=l\circ w\circ f}\\
P_0\otimes_CB\arrow{d}{w}\\
P_0\arrow{r}{l} & M
\end{tikzcd}
\]
This gives $h=l\circ w\circ f$ and we conclude that $\overline{f}$ is surjective.
  \end{proof}
 
 For an illustration of this theorem, see Examples $\ref{Extfirst}$ and $\ref{Extsecond}$ in section 5.
 \subsection{Corollaries}
 We now examine several corollaries of our main result.  For the first corollary, we say $M$ is a $\it{partial~cotilting~module}$ if $\id_CM\leq1$ and $\Ext_C^1(M,M)=0$ and $\it{cotilting}$ if the number of pairwise, non-isomorphic, indecomposable summands of $M$ equals the number of isomorphism classes of simple $C$-modules.
 \begin{cor} 
 \label{Tilt Ext}
 If M is a partial tilting or cotilting C-module, then M is a rigid B-module.
 \end{cor}
 \begin{proof}
 We assume $M$ is a partial tilting module.  The proof for the case $M$ is a partial cotilting module is dual.  Since $\pd_CM\leq1$, we have that Lemma $\ref{Homological Result}$ implies $\text{Hom}_C(\tau_C^{-1}\Omega_C^{-1}C,M)=0$.
 The statement now follows from Theorem $\ref{Main Ext}$. 
 \end{proof}
 The next result holds in the specific case where $C$ is tilted and $B$ is cluster-tilted.
 \begin{cor} 
 \label{Ext Cor 1}
 Let C be a tilted algebra with B the corresponding cluster-titled algebra.  Suppose M is an indecomposable, rigid C-module.  Then M is a rigid B-module.
 \end{cor}
 \begin{proof}
 Let $M$ be an indecomposable, rigid $C$-module.  By Proposition $\ref{Tilting Prop}$(b), we have that $\pd_CM\leq1$ or $\id_CM\leq1$.  Since $M$ is rigid, we have $M$ is partial tilting or partial cotilting.  By Corollary $\ref{Tilt Ext}$, our statement follows.  
 \end{proof}
 We now state the converse to Theorem $\ref{Main Ext}$.  We note that if $M$ is a $C$-module which is rigid as a $B$-module, then $M$ is trivially a rigid $C$-module.   
 \begin{prop} Assume $C$ is an algebra of global dimension 2.  Let M be a C-module with a projective cover $g\!:P_0\rightarrow M$ and an injective envelope $h\!:M\rightarrow I_0$ in $\mathop{\emph{mod}}C$.  Suppose M is a rigid B-module.
 \begin{enumerate}
 \item[$\emph{(a)}$] If $\emph{Ext}_B^1(P_0,M)=0$, then $\emph{Hom}_C(\tau_C^{-1}\Omega_C^{-1}P_0,M)=0$. 
 
 \item[$\emph{(b)}$] If $\emph{Ext}_B^1(M,I_0)=0$, then $\emph{Hom}_C(M,\tau_C\Omega_C I_0)=0$.
 \end{enumerate}
 \end{prop}
 \begin{proof}
 We prove case (a) with case (b) being dual.  Consider the following sequence in $\mathop{\text{mod}}B$ guaranteed by Proposition $\ref{SS, Proposition 3.6}$ and Proposition $\ref{SS Prop 4.1}$.
 \[
 0\rightarrow\tau_C^{-1}\Omega_C^{-1}P_0\xrightarrow{f}P_0\otimes_CB\rightarrow P_0\rightarrow 0.
 \]
 Apply $\text{Hom}_B(-,M)$ to obtain
 \[
 0\rightarrow\text{Hom}_B(P_0,M)\rightarrow\text{Hom}_B(P_0\otimes_CB,M)\xrightarrow{\overline{f}}\text{Hom}_B(\tau_C^{-1}\Omega_C^{-1}P_0,M)\rightarrow\text{Ext}_B^1(P_0,M).
 \]
 Since the sequence is exact and $\text{Ext}_B^1(P_0,M)=0$ by assumption, we have that $\overline{f}$ is surjective.  This implies that any morphism of $B$-modules, $j\!:\tau_C^{-1}\Omega_C^{-1}P_0\rightarrow M$, factors through the projective $B$-module $P_0\otimes_CB$.  Since $g\!:P_0\rightarrow M$ is a surjective morphism, there exists a morphism $k\!:\tau_C^{-1}\Omega_C^{-1}P_0\rightarrow P_0$ such that $j=g\circ k$.  
 \[
\begin{tikzcd}
& \tau_C^{-1}\Omega_C^{-1}P_0\arrow{d}{j=g\circ k}\arrow{dl}[swap]{k}\\
P_0\arrow{r}{g} & M
\end{tikzcd}
\] 
 
But $\pd_CP_0=0$ and Lemma $\ref{Homological Result}$ implies $\text{Hom}_B(\tau_C^{-1}\Omega_C^{-1}P_0,P_0)=0$.  Thus $k$ must be the 0 morphism.  This forces $j$ to also be the 0 morphism.  Since $j$ was arbitrary we conclude that $\text{Hom}_B(\tau_C^{-1}\Omega_C^{-1}P_0,M)=0$ which further implies $\text{Hom}_C(\tau_C^{-1}\Omega_C^{-1}P_0,M)=0$ by restriction of scalars. 
\end{proof}

\section{Examples}
In this section we illustrate our main results with several examples.  We will use the following throughout this section.  Let $A$ be the path algebra of the following quiver:
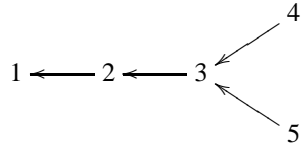
\[
\xymatrix@R10pt{&&&4\ar[dl]\\1&2\ar[l]&3\ar[l]\\ &&&5\ar[ul]}
\]
\par
Since $A$ is a hereditary algebra, we may construct a tilted algebra.  To do this, we need an $A$-module which is tilting.  Consider the Auslander-Reiten quiver of $A$ which is given by:
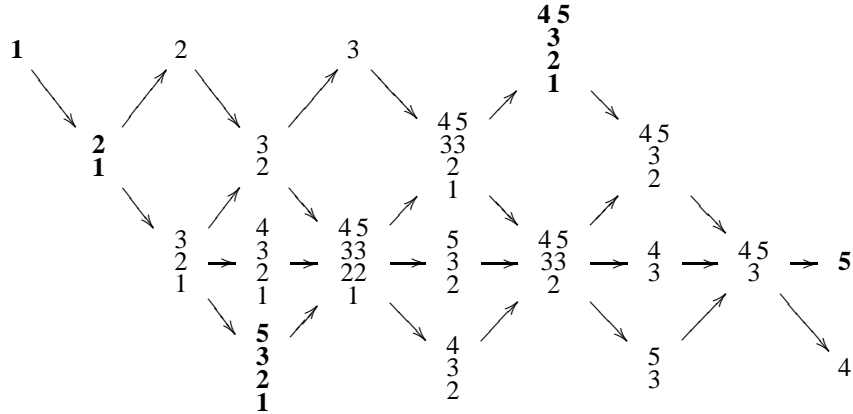
\[
\xymatrix@C=10pt@R=0pt{
{\begin{array}{c} \bf 1\end{array}}\ar[dr]&&
{\begin{array}{c} 2\end{array}}\ar[dr]&&
{\begin{array}{c} 3\end{array}}\ar[dr]&&
{\begin{array}{c}  \bf 4\ 5\\  \bf 3\\ \bf 2\\ \bf 1 \end{array}}\ar[dr]&&
\\
&{\begin{array}{c}  \bf 2\\ \bf 1\end{array}}\ar[dr]\ar[ur]&&
{\begin{array}{c} 3\\2\end{array}}\ar[dr]\ar[ur]&&
{\begin{array}{c} 4\,5\\33\\2\\1\end{array}}\ar[dr]\ar[ur]&&
{\begin{array}{c} 4\,5\\3\\2\end{array}}\ar[dr]&&
\\
&&{\begin{array}{c} 3\\2\\1\end{array}}\ar[dr]\ar[ur]\ar[r]&
{\begin{array}{c} 4\\3\\2\\1\end{array}}\ar[r]&
{\begin{array}{c} 4\,5\\33\\22\\1\end{array}}\ar[dr]\ar[ur]\ar[r]&
{\begin{array}{c} 5\\3\\2\end{array}}\ar[r]&
{\begin{array}{c} 4\,5\\33\\2\end{array}}\ar[dr]\ar[ur]\ar[r]&
{\begin{array}{c} 4\\3\end{array}}\ar[r]&
{\begin{array}{c} 4\,5\\3\end{array}}\ar[dr]\ar[r]&
{\begin{array}{c}  \bf 5\end{array}}
\\
&&&{\begin{array}{c} \bf  5\\ \bf 3\\ \bf 2\\ \bf 1\end{array}}\ar[ur]&&
{\begin{array}{c} 4\\3\\2\end{array}}\ar[ur]&&
{\begin{array}{c} 5\\3\end{array}}\ar[ur]&&
{\begin{array}{c} 4\end{array}}
}
\]

Let $T$ be the tilting $A$-module
\[
T=
{\begin{array}{c} 5\end{array}} \oplus
{\begin{array}{c} 4\,5\\3\\2\\1\end{array}} \oplus
{\begin{array}{c} 5\\3\\2\\1\end{array}} \oplus
{\begin{array}{c} 2\\1\end{array}} \oplus
{\begin{array}{c} 1\end{array}} 
\]

The corresponding titled algebra $C=\text{End}_AT$ is given by the bound quiver
$$\begin{array}{cc}
\xymatrix{1\ar[r]^\za&2\ar[r]^\zb&3\ar[r]^\zg&4\ar[r]&5}
&\quad\za\zb\zg=0\end{array}$$

Then, the Auslander-Reiten quiver of $C$ is given by: 

$$
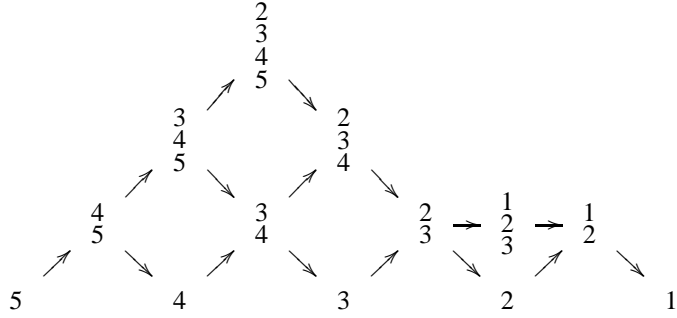
\xymatrix@C=10pt@R=0pt
{&&&{\begin{array}{c} 2\\ 3\\4\\5 \end{array}}\ar[dr]&&
&&
\\
&&{\begin{array}{c} 3\\4\\5 \end{array}}\ar[ur]\ar[dr]&&
{\begin{array}{c} 2\\ 3\\4 \end{array}}\ar[dr]&
\\
&{\begin{array}{c} 4\\5  \end{array}}\ar[ur]\ar[dr]&&
{\begin{array}{c} 3\\ 4  \end{array}}\ar[ur]\ar[dr]&&
{\begin{array}{c} 2\\3  \end{array}}\ar[r]\ar[dr]&
{\begin{array}{c} 1\\2\\ 3 \end{array}}\ar[r]&
{\begin{array}{c} 1\\2 \end{array}}\ar[dr]&&
\\
{\begin{array}{c}\ \\5\\ \  \end{array}}\ar[ur]&&
{\begin{array}{c}\ \\ 4\\ \  \end{array}}\ar[ur]&&
{\begin{array}{c}\ \\ 3\\ \  \end{array}}\ar[ur]&&
{\begin{array}{c}\ \\ 2\\ \  \end{array}}\ar[ur]&&
{\begin{array}{c}\ \\ 1\\ \  \end{array}}&&
}$$

The corresponding cluster-tilted algebra $B=C\ltimes\text{Ext}_C^2(DC,C)$ is given by the bound quiver 
$$\begin{array}{cc}
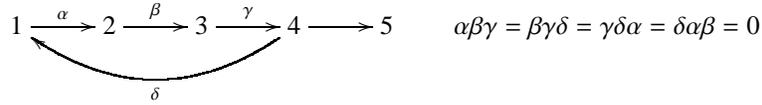
\xymatrix{1\ar[r]^\za&2\ar[r]^\zb&3\ar[r]^\zg&4\ar@/^20pt/[lll]^\zd\ar[r]&5}
&\quad \za\zb\zg=\zb\zg\zd=\zg\zd\za=\zd\za\zb=0 
\end{array}$$

Then, the Auslander-Retien quiver of $B$ is given by:

$$
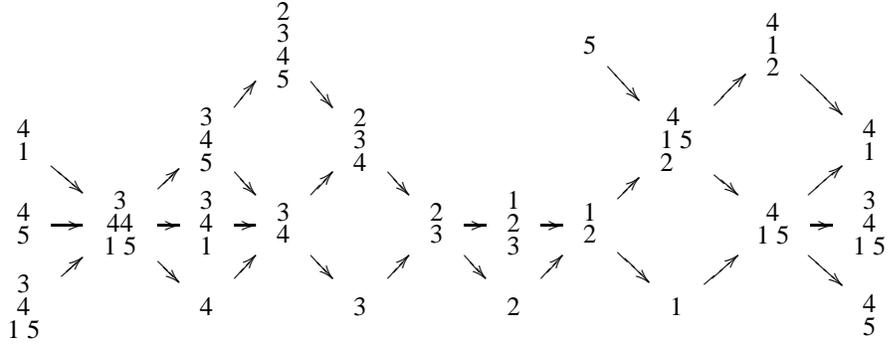
\xymatrix@C=8pt@R=0pt
{&&{\begin{array}{c} \end{array}}
&&{\begin{array}{c} 2\\ 3\\4\\5 \end{array}}\ar[dr]&&
{\begin{array}{c} \end{array}}&& 
{\begin{array}{c} 5 \end{array}}\ar[dr]&&
{\begin{array}{c} 4\\ 1\\2 \end{array}}\ar[dr]&&
{\begin{array}{c} \end{array}}
\\
&{\begin{array}{c} 4\\1 \end{array}}\ar[dr]&&
{\begin{array}{c} 3\\4\\5 \end{array}}\ar[ur]\ar[dr]&&
{\begin{array}{c} 2\\3\\4 \end{array}}\ar[dr]&&
{\begin{array}{c} \end{array}}&&
{\begin{array}{l} \ 4\\1\ 5\\2 \end{array}}\ar[ur]\ar[dr]&&
{\begin{array}{c} 4\\1 \end{array}}&&
\\
&{\begin{array}{c} 4\\ 5 \end{array}}\ar[r]&
{\begin{array}{c} 3\\44\\ 1\ 5 \end{array}}\ar[ur]\ar[dr]\ar[r]&
{\begin{array}{c} 3\\4\\1 \end{array}}\ar[r]&
{\begin{array}{c} 3\\4 \end{array}}\ar[ur]\ar[dr]&
{\begin{array}{c} \end{array}}&
{\begin{array}{c} 2\\ 3 \end{array}}\ar[dr]\ar[r]&
{\begin{array}{c} 1\\2\\ 3 \end{array}}\ar[r]&
{\begin{array}{c} 1\\2 \end{array}}\ar[dr]\ar[ur]&
{\begin{array}{c} \end{array}}&
{\begin{array}{c} 4\\1\ 5 \end{array}}\ar[dr]\ar[r]\ar[ur]&
{\begin{array}{c} 3\\4\\1\ 5 \end{array}}&&
\\
&{\begin{array}{c}3 \\ 4\\ 1\ 5 \end{array}}\ar[ur]
&&{\begin{array}{c}\ \\ 4\\ \  \end{array}}\ar[ur]&&
{\begin{array}{c}\ \\ 3\\ \  \end{array}}\ar[ur]&&
{\begin{array}{c}\ \\ 2\\ \  \end{array}}\ar[ur]&&
{\begin{array}{c}\ \\ 1\\ \  \end{array}}\ar[ru]&&
{\begin{array}{c}\ \\ 4\\ 5  \end{array}}
}$$

We wish to illustrate Theorem $\ref{Main Homo}$ with an example for each case.  We will use Lemma $\ref{Homological Result}$ frequently so we note that 
\[
\tau_C^{-1}\Omega_C^{-1}C={\begin{array}{c}1\\2\end{array}}\oplus{\begin{array}{c}1\end{array}}~~,~~
\tau_C\Omega_C(DC)={\begin{array}{c}3\\4\end{array}}\oplus{\begin{array}{c}4\end{array}}.
\]

\begin{example}
\label{Homofirst}
We'll start with the projective dimension equal to 2.  In $\mathop{\text{mod}}C$, consider the module $M=\begin{array}{c}1\\2\end{array}$.  Since $\text{Hom}_C(\tau_C^{-1}\Omega_C^{-1}C,M)\neq0$,  Lemma $\ref{Homological Result}$ says $\pd_CM=2$.  Thus, Theorem $\ref{Main Homo}$ says $\pd_BM=\infty$ and we have the following projective resolution in $\mathop{\text{mod}}B$
\[
\dots\rightarrow
{\begin{array}{l}\ 4\\1\ 5\\2\end{array}}\rightarrow
{\begin{array}{c}3\\4\\1\,5\end{array}}\rightarrow
{\begin{array}{c}1\\2\\3\end{array}}\rightarrow
{\begin{array}{c}1\\2\end{array}}\rightarrow
{\begin{array}{c}0\end{array}}.
\] 
\end{example}
\begin{example}
\label{Homosecond}
Next, let's examine the projective case, i.e., projective dimension equal to 0.  In $\mathop{\text{mod}}C$, consider the module $M=5$.  Then $M$ is the projective $C$-module at vertex 5.  Since $\text{Hom}_C(M,\tau_C\Omega_C(DC))=0$, Lemma $\ref{Homological Result}$ says $\text{id}_CM\leq1$.  Thus, Theorem $\ref{Main Homo}$ says $\pd_BM=0$.  Now, consider $N={\begin{array}{c}4\\5\end{array}}$ in $\mathop{\text{mod}}C$. Then $N$ is a projective $C$-module.  Now, we have that $\text{Hom}_C(N,\tau_C\Omega_C(DC))\neq0$.  Thus,  Lemma $\ref{Homological Result}$ says $\id_CN=2$.  Finally, Theorem $\ref{Main Homo}$ states $N={\begin{array}{c}4\\5\end{array}}$ is not a projective $B$-module and $\pd_BN=\infty$ with the following projective resolution in $\mathop{\text{mod}}B$
\[
\dots\rightarrow
{\begin{array}{c} 3\\4\\1\,5\end{array}}\rightarrow
{\begin{array}{c}1\\2\\3\end{array}}\rightarrow
{\begin{array}{l} 4\\1\ 5\\2\end{array}}\rightarrow
{\begin{array}{c}4\\5\end{array}}\rightarrow
{\begin{array}{c}0\end{array}}.
\] 
\end{example}
\begin{example} 
\label{Homothird}
Finally, let's examine the case where the projective dimension is equal to 1.  Consider the $C$-module $M={\begin{array}{c}3\\4\end{array}}$.  Since $\text{Hom}_C(\tau_C^{-1}\Omega_C^{-1}C,M)=0$, Lemma $\ref{Homological Result}$ says $\pd_CM\leq1$ with projective resolution
\[
{\begin{array}{c} 0\end{array}}\rightarrow
{\begin{array}{c}5\end{array}}\rightarrow
{\begin{array}{c}3\\4\\5\end{array}}\rightarrow
{\begin{array}{c} 3\\4\end{array}}\rightarrow
{\begin{array}{c}0\end{array}}.
\]
Denote $P_1=5$ and $P_0={\begin{array}{c}3\\4\\5\end{array}}$.  Since $\text{Hom}_C(M,\tau_C\Omega_C(DC))\neq0$, Lemma $\ref{Homological Result}$ says $\id_CM=2$.  Also, note that $\tau_C^{-1}\Omega_C^{-1}\,P_1\not\cong\tau_C^{-1}\Omega_C^{-1}\,P_0$ because $\tau_C^{-1}\Omega_C^{-1}\,P_1=0$ while $\tau_C^{-1}\Omega_C^{-1}\,P_0={\begin{array}{c}1\\2\end{array}}$.  Thus, Theorem $\ref{Main Homo}$ says $\pd_BM=\infty$ and we have the following projective resolution in $\mathop{\text{mod}}B$
\[
\dots\rightarrow
{\begin{array}{c} 2\\3\\4\\5\end{array}}\rightarrow
{\begin{array}{c}5\end{array}}\oplus
{\begin{array}{c}1\\2\\3\end{array}}\rightarrow
{\begin{array}{c} 3\\4\\1\,5\end{array}}\rightarrow
{\begin{array}{c}3\\4\end{array}}\rightarrow
{\begin{array}{c}0\end{array}}.
\]
Next, consider the $C$-module $N={\begin{array}{c}2\\3\\4\end{array}}$.  Since $\text{Hom}_C(\tau_C^{-1}\Omega_C^{-1}C,N)=0$, Lemma $\ref{Homological Result}$ says that $\pd_CN=1$ with minimal projective resolution
\[
{\begin{array}{c} 0\end{array}}\rightarrow
{\begin{array}{c}5\end{array}}\rightarrow
{\begin{array}{c}2\\3\\4\\5\end{array}}\rightarrow
{\begin{array}{c}2\\ 3\\4\end{array}}\rightarrow
{\begin{array}{c}0\end{array}}.
\]
Denote $P_1^{\prime}=5$ and $P_0^{\prime}={\begin{array}{c}2\\3\\4\\5\end{array}}$.  Since $\text{Hom}_C(N,\tau_C\Omega_C(DC))=0$, Lemma $\ref{Homological Result}$ says that $\text{id}_CN\leq1$.  Also, note that $\tau_C^{-1}\Omega_C^{-1}\,P_1^{\prime}\cong\tau_C^{-1}\Omega_C^{-1}\,P_0^{\prime}=0$.  Thus, Theorem $\ref{Main Homo}$ says $\pd_BN=1$ and Corollary $\ref{Resolution}$ implies the minimal projective resolution in $\mathop{\text{mod}}C$
is the same as the minimal projective resolution in $\mathop{\text{mod}}B$. 

\end{example}

We will illustrate Theorem $\ref{Main Ext}$ with two examples.
\begin{example}
\label{Extfirst}
Consider the $C$-module $M={\begin{array}{c}4\\5\end{array}}\oplus{\begin{array}{c}1\\2\end{array}}$.  To use Theorem $\ref{Main Ext}$ we need several preliminary calculations.  We have a projective cover and an injective envelope in $\mathop{\text{mod}}C$
\[
f\!:{\begin{array}{c}4\\5\end{array}}\oplus{\begin{array}{c}1\\2\\3\end{array}}\rightarrow M~~,~~g\!:M\rightarrow{\begin{array}{c}2\\3\\4\\5\end{array}}\oplus{\begin{array}{c}1\\2\end{array}}.
\]
Let us denote ${\begin{array}{c}4\\5\end{array}}\oplus{\begin{array}{c}1\\2\\3\end{array}}$ by $P_0$ and${\begin{array}{c}2\\3\\4\\5\end{array}}\oplus{\begin{array}{c}1\\2\end{array}}$ by $I_0$.  Then we have
\[ 
\tau_C^{-1}\Omega_C^{-1}P_0={\begin{array}{c}1\\2\end{array}}~~,~~
\tau_C\Omega_CI_0={\begin{array}{c}4\end{array}}.
\]
It is easily seen that $\text{Ext}_C^1(M,M)=0$ but $\text{Ext}_B^1(M,M)\neq0$ with self-extension
\[
{\begin{array}{c}0\end{array}}\rightarrow
{\begin{array}{c}1\\2\end{array}}\rightarrow
{\begin{array}{l} 4\\1\ 5\\2\end{array}}\rightarrow
{\begin{array}{c}4\\5\end{array}}\rightarrow
{\begin{array}{c}0\end{array}}
\]
Note that $\text{Hom}_C(\tau_C^{-1}\Omega_C^{-1}P_0,M)\neq0$ and $\text{Hom}_C(M,\tau_C\Omega_CI_0)\neq0$ in accordance with Theorem $\ref{Main Ext}$.

\end{example}
\begin{example}
\label{Extsecond}
Consider the $C$-module $N={\begin{array}{c}5\end{array}}\oplus{\begin{array}{c}3\end{array}}$.  We have a projective cover and an injective envelope in $\mathop{\text{mod}}C$ 
\[
f\!:{\begin{array}{c}5\end{array}}\oplus{\begin{array}{c}3\\4\\5\end{array}}\rightarrow M~~,~~g\!:M\rightarrow{\begin{array}{c}2\\3\\4\\5\end{array}}\oplus{\begin{array}{c}1\\2\\3\end{array}}.
\]
Denote ${\begin{array}{c}5\end{array}}\oplus{\begin{array}{c}3\\4\\5\end{array}}$ by $P_0$ and ${\begin{array}{c}2\\3\\4\\5\end{array}}\oplus{\begin{array}{c}1\\2\\3\end{array}}$ by $I_0$.  Then we have
\[
\tau_C^{-1}\Omega_C^{-1}P_0=1~~,~~\tau_C\Omega_CI_0=0.
\]
Now, we have $\text{Ext}_C^1(M,M)=0$ and $\text{Ext}_B^1(M,M)=0$.  Note, $\text{Hom}_C(\tau_C^{-1}\Omega_C^{-1}P_0,M)=0$ and $\text{Hom}_C(M,\tau_C\Omega_CI_0)=0$ in accordance with Theorem $\ref{Main Ext}$.
\end{example}
\section*{Acknowledgements}
The author gratefully acknowledges support from the University of Connecticut and would like to thank Ralf Schiffler for his guidance and helpful comments.

\noindent Department of Mathematics, University of Connecticut, Storrs, CT 06269-3009, USA
\it{E-mail address}: \bf{stephen.zito@uconn.edu}


\begin{thebibliography}{1}

\bibitem{A} C. Amoit, Cluster categories for algebras of global dimension 2 and quivers with potential, $\emph{Ann. Inst. Fourier}~\bf{59}$, (2009), no. 6, 2525--2590.


\bibitem{ABS} I. Assem, T. Br$\ddot{\text{u}}$stle and R. Schiffler, Cluster-tilted algebras as trivial extensions, $\emph{Bull. Lond. Math. Soc.}~\bf{40}$ (2008), 151--162.

\bibitem{ABS2} I. Assem, T. Br$\ddot{\text{u}}$stle and R. Schiffler, Cluster-tilted algebras and slices, $\emph{J. of}$ 
\newline
$\emph{Algebra}~\bf{319}$ (2008), 3464--3479.

\bibitem{ABS3} I. Assem, T. Br$\ddot{\text{u}}$stle and R. Schiffler, On the Galois covering of a cluster-tilted algebra, $\emph{J. Pure Appl. Alg.}~\bf{213}$ (7) (2009), 1450--1463.

\bibitem{ABS4} I. Assem, T. Br$\ddot{\text{u}}$stle and R. Schiffler, Cluster-tilted algebras without clusters
\newline
$\emph{J. Algebra}~\bf{324}$, (2010), 2475--2502. 

\bibitem{AM} I. Assem and N. Marmaridis, Tilting modules and split-by-nilpotent extensions, $\emph{Comm. Algebra}~\bf{26}$ (1998), 1547--1555.
\bibitem{ASS} I. Assem, D. Simson and A. Skowronski, \emph{Elements of the Representation Theory of}
\emph{Associative Algebras, 1: Techniques of Representation Theory}, London Mathematical Society Student Texts 65, Cambridge University Press, 2006.

\bibitem{AZ} I. Assem and D. Zacharia, Full embeddings of almost split sequences over split-by-nilpotent extensions, $\emph{Coll. Math.}~\bf{81}$, (1) (1999), 21--31.

 \bibitem{BFPPT} M. Barot, E. Fernandez, I. Pratti, M. I. Platzeck and S. Trepode, From iterated tilted to cluster-tilted algebras, $\emph{Adv. Math.}~\bf{223}$ (2010), no. 4, 1468--1494.

\bibitem{BBT} L. Beaudet, T. Brustle and G. Todorov, Projective dimension of modules over cluster-tilted algebras,      $\emph{Algebr. and Represent. Theory}~\bf{17}$ (2014), no. 6, 1797--1807.

\bibitem{BOW} M. A. Bertani-${\O}$kland, S. Oppermann and A. Wr$\mathring{\text{a}}$lsen, Constructing tilted algebras from cluster-tilted algebras, $\emph{J. Algebra}~\bf{323}$ (2010), no. 9, 2408--2428.


\bibitem{BMMRRT} A. B. Buan, R. Marsh, M. Reineke, I. Reiten and G. Todorov, Tilting theory and cluster combinatorics, $\emph{Adv. Math.}~\bf{204}$ (2006), no. 2, 572--618.

\bibitem{BMR} A. B. Buan, R. Marsh and I. Reiten, Cluster-tilted algebras, \emph{Trans. Amer. Math. Soc.}~
$\bf{359}$ (2007), no. 1, 323-332.

\bibitem{BMR2} A. B. Buan, R. Marsh and I. Reiten, Cluster-tilted algebras of finite representation type, $\emph{J. Algebra}~\bf{306}$ (2006), no. 2, 412--431.
\bibitem{BMR3} A. B. Buan, R. Marsh and I. Reiten, Cluster mutation via quiver representations, $\emph{Comment. Math. Helv.}~\bf{83}$ (2008), no. 1, 143--177.

                 



\bibitem{CCS} P. Caldero, F. Chapoton and R. Schiffler, Quivers with relations arising from clusters ($A_n$ case), $\emph{Trans. Amer. Math. Soc.}~\bf{358}$ (2006), no. 4, 359--376.

\bibitem{CCS2} P. Caldero, F. Chapoton and R. Schiffler, Quivers with relations and cluster tilted algebras, $\emph{Algebr. and Represent. Theory}~\bf{9}$, (2006), no. 4, 359--376.


\bibitem{FZ} S. Fomin and A. Zelevinsky, Cluster algebras I: Foundations, $\emph{J.  Amer.  Math. Soc.}\\~\bf{15}$ (2002), 497--529.


\bibitem{K} B. Keller, On triangulated orbit categories, $\emph{Documenta Math.}~\bf{10}$ (2005), 551--581.

\bibitem{KR} B. Keller and I. Reiten, Cluster-tilted algebras are Gorenstein and stably Calabi-Yau, $\emph{Adv. Math.}~\bf{211}$ (2007), no. 1, 123--151.

\bibitem{P} P. G. Plamondon, Cluster algebras via cluster categories with infinite-dimensional morphism spaces, $\emph{Compos. Math.}~\bf{147}$ (2011), no. 6, 1921--1954.


\bibitem{SS} R. Schiffler and K. Serhiyenko, Induced and coinduced modules in cluster-tilted algebras, preprint, $\bf{arXiv:1410.1732}$.

\end{thebibliography}
\end{document}